\documentclass{amsproc}
\usepackage{geometry}                
\usepackage{graphicx}
\usepackage{amssymb}
\usepackage{epstopdf}

\newtheorem{theorem}{Theorem}[section]

\newtheorem{lemma}[theorem]{Lemma}
\newtheorem{corollary}[theorem]{Corollary}

\def\w{\hat{w}}

\def\C{{\mathcal C}}

\def\u{\hat{u}}
\def\a{\alpha}

\def\dim{{\rm dim \,}}

\def\domain{domain}
\def\range{range}

\def\min{{\rm min}}

\def\length{length}

\def\Z{{\mathbb{Z}}}

\title{The Classification of $\Z_p$-Modules with Partial Decomposition Bases in $L_{\infty \omega}$}
\author{Carol Jacoby}
\address{Jacoby Consulting, Long Beach, California}
\email{cjacoby@jacobyconsulting.com}

\author{Peter Loth}
\address{Department of Mathematics, Sacred Heart University, Fairfield, CT 06825}
\email{lothp@sacredheart.edu}


\begin{document}
\maketitle

\begin{abstract}
Ulm's Theorem presents invariants that classify countable abelian torsion groups up to isomorphism. Barwise and Eklof extended this result to the classification of arbitrary abelian torsion groups up to $L_{\infty \omega}$-equivalence. 
In this paper, we extend this classification to a class of mixed $\Z_p$-modules which includes all Warfield modules and is closed under $L_{\infty\omega}$-equivalence.
The defining property of these modules is the existence of what we call a partial decomposition basis, a generalization of the concept of decomposition basis. 
We prove a complete classification theorem in $L_{\infty\omega}$ using invariants deduced from the classical Ulm and Warfield invariants.
\end{abstract}

\section{Introduction}
This is the first of two papers based on a 1980 doctoral dissertation \cite{J1} that has not been previously published in a readily available form but has been the starting point for recent work, specifically \cite{JLLS}, \cite{JL1}, \cite{J2}, \cite{JL2}, \cite{JL3}, and \cite{JL4}. Independently, G\"obel, Leistner, Loth and Str\"ungmann \cite{GLLS} recently explored the same topic and proved similar results.
Here we focus on the local case. A forthcoming paper will address the global case \cite{JL5}. This paper corrects and clarifies the original and streamlines some of the proofs.

Ulm's Theorem \cite{U} defines invariants that classify countable torsion abelian groups up to isomorphism. Generalizations of this theorem have taken two directions. The first extends the class of groups that may be classified up to isomorphism. Warfield \cite{W2} developed new invariants that, along with the Ulm invariants, serve to classify a class of local groups including all countable torsion and all completely decomposable torsion-free local abelian groups. The other direction was taken by Barwise and Eklof \cite{BE}. They looked at the classification problem in the language $L_{\infty \omega}$ and classified all torsion abelian groups up to $L_{\infty \omega}$-equivalence using modified Ulm invariants. Since countable groups that are $L_{\infty \omega}$-equivalent are also isomorphic, the result is indeed a generalization of Ulm's Theorem. 

This paper seeks to unify these two generalizations of Ulm's Theorem by defining a class of local groups that includes those studied by Warfield and that may be classified in $L_{\infty \omega}$. The defining property of these groups is the existence of what we call a partial decomposition basis, a generalization of the concept of decomposition basis that is preserved under $L_{\infty \omega}$-equivalence. Invariants are defined for this class and a classification theorem is proved. As in most problems of this type, we first look at the local case. 
The global case will be addressed in a future paper \cite{JL5}.

Section 2 presents the background material, including the definition and properties of $L_{\infty \omega}$, the Ulm invariants and Ulm's Theorem, the result of Barwise and Eklof, the work of Warfield and others, and the definition and properties of decomposition bases. It is assumed that the reader is familiar with first order logic and the fundamentals of the theories of abelian groups and modules over a principal ideal domain. The word ``group" used in this paper will mean abelian group and ``rank" will mean torsion-free rank, i.e., the dimension of $\mathbb Q \otimes  G$ as a vector space over $\mathbb Q$.

The third section introduces the concept of partial decomposition basis and defines the restricted Warfield invariant. It is proved that this invariant is independent of the choice of partial decomposition basis and is indeed invariant under $L_{\infty \omega}$-equivalence.

Section 4 proves the classification theorem for local groups with partial decomposition bases in $L_{\infty\omega}$.

\section{Background}

\subsection{The Language $L_{\infty \omega}$ and Partial Isomorphisms}

The results of this paper will be considered in light of the language of infinitary logic known as $L_{\infty \omega}$. This is an extension of the familiar language of first order logic to allow infinite conjunctions and disjunctions.

The symbols of the language include those of first order logic with identity and a variable $v_\alpha$ for each ordinal $\alpha$, and since we will be using this language to discuss abelian groups, we include the binary function symbol +, the unary function symbol - and the constant 0. We use these symbols to form atomic formulas in the usual way.

Now for each ordinal $\alpha$ we define, by induction on $\alpha$, $L_{\infty \omega}^\alpha$ to be the smallest class $Y$ containing the atomic formulas and such that

\begin{itemize}
    \item [\rm{(i)}] if $\varphi \in Y$ then $\neg \varphi \in Y$
    \item [\rm{(ii)}] if $\Phi$ is a set, $\Phi \subseteq Y$, then $\wedge \Phi$ and $\vee \Phi$ $\in Y$
    \item [\rm{(iii)}] if $\varphi \in L_{\infty \omega}^\beta$ for some $\beta < \alpha$ and $v$ is a variable, then $\forall v \varphi \in Y$ and $\exists v \varphi \in Y$
    \end{itemize}
Then we let $L_{\infty \omega} = \bigcup_{\alpha \in Ord} {L_{\infty \omega}^\alpha}$.

An example of a sentence of $L_{\infty \omega}$ is $$\psi = \forall x \vee \{nx = 0: n \in \omega , n \neq 0 \}$$ where $nx$ is an abbreviation of the expression representing $x + \cdots +x$ with $n$ terms. It is clear that if $G$ is a group, then $G$ is torsion if and only if $G \models \psi$. There is no sentence of first order logic with this property, as can be seen by a simple compactness argument. 

If $G$ and $H$ are groups, we say $G \equiv_\infty  H$, $G$ and $H$ are {\em $L_{\infty \omega}$-equivalent}, if and only if $G$ and $H$ satisfy the same sentences of  $L_{\infty \omega}$. This relationship is clearly stronger than elementary equivalence, but weaker than isomorphism. $G$ and $H$ do, however, share isomorphisms on subgroups in the following sense.

We say that $G \cong_p H$, $G$ is {\em partially isomorphic} to $H$, if there is a nonempty set $I$ of isomorphisms between subgroups of $G$ and subgroups of $H$ with the {\em back-and-forth property}: For any $f \in I$ and $a \in G$ (respectively $b \in H$), there is a $g \in I$ such that $f \subseteq g$ and $a \in \domain(g)$ (respectively $b \in \range(g)$). We will sometimes write $I: G \cong_p H$.

\begin{theorem} \label{thm1}
If $G$ and $H$ are countable and $G \cong_p H$ then $G \cong H$.
\end{theorem}

The proof is in Barwise \cite[Theorem 2]{Bar} and simply involves starting with any $f \in I: G \cong_p H$ and extending it alternately to elements of $H$ and $G$.

Thus far we have defined two concepts, the logical  concept of equivalence in the language $L_{\infty \omega}$ and the algebraic concept of partial isomorphism. Karp's Theorem allows us to use them interchangeably.

\begin{theorem} \label{thm2} \cite{Kar}
$G \equiv_\infty H$ if and only if $G \cong_p H$.
\end{theorem}

The proof is in Barwise \cite[Theorem 3]{Bar}, and also shows that if $G \cong_p H$, we may choose $I: G \cong_p H$ such that every $f \in I$ has finitely generated domain and range.

\subsection{Height}

In this paper we consider modules over a principal ideal domain $R$. This allows us to consider groups as modules over $\Z$ and local groups as modules over $\Z_p$, where $\Z_p$ is the ring of integers localized at the prime $p$.
In the following definitions, we fix a module $G$ and a prime $p$ in $R$. Let $pG = \{px: x \in G\}$.
If $\alpha$ is an ordinal, we define $p^\alpha G$ by induction on $\alpha$ as follows: 
\[p^\alpha G = 
\begin{cases}
p(p^\beta G) &\text{ if } \alpha = \beta + 1
\\
 \bigcap_{\beta < \alpha} p^\beta G &\text{ if } \alpha \text{ is a limit ordinal.}
\end{cases}
\]
We define the {\em p-height} of $x$, $|x|_p$, for $x \in G$, to be the unique ordinal $\alpha$ such that $x \in p^\alpha G$ and $x \notin p^{\alpha + 1} G$ if it exists, and the symbol $\infty$ otherwise. Note that the value of $|x|_p$ is dependent on the module $G$, but the module under consideration will be clear from the context. We may talk about the height of $x$ and write $|x|$ if $p$ is understood.
Then, if $S$ is a submodule of $G$ we say $x\in G$ is {\em proper} with respect to $S$ is $x$ has maximal height in the coset $x+S$.  We define the {\em ($p$-)length} of $G$, written {\em length$(G)$}, to be the least ordinal $\alpha$ such that $p^\alpha G=p^{\alpha +1} G$. 
Such an $\alpha$ exists by cardinality considerations. We say that a function $f:G\to H$ {\em preserves ($p$-)heights} if for every $x \in \domain(f)$, $|x|=|f(x)|$, and it {\em preserves heights up to $\alpha$} if for every $x \in \domain(f)$, if $|x| < \alpha$ then $|f (x)| = |x|$ and if $|x| \ge \alpha$ then $|f (x)| \ge \alpha$.
If $R=\Z_p$, we call $p^\infty G = \cap p^\a G$ {\em the divisible part} of $G$ and say $G$ is {\em reduced} if $p^\infty G =0$. 

We say a sequence $(\alpha_i)$, $ i \in \omega$, is an {\em Ulm sequence} if each $\alpha_i$ is either an ordinal or the symbol $\infty$ and for all $i$, if $\alpha_i = \infty$, then $\alpha_{i+1} = \infty$, and if $\alpha_i \ne \infty$, then $\alpha_{i+1} > \alpha_i$.  If $x$ is an element of the module $G$, $U(x)$, the {\em Ulm sequence of $x$}, is the sequence $(|p^ix|)_{i \in \omega}$.  We call the Ulm sequences $(\alpha_i)$ and $(\beta_i)$ {\em equivalent}, written $(\alpha_i) \sim (\beta_i)$ if there are positive integers $m$ and $n$ such that $\alpha_{i+n} = \beta_{i+m}$ for all $i \ge 0$.  Note that if $x$ and $y$ have a nonzero common multiple, then $U(x)$ and $U(y)$ are equivalent. In particular, if $G$ is of rank one, the equivalence class of $U(x)$ is independent of the choice of $x$ of infinite order.  Kaplansky \cite{Kap1} proved that this invariant classifies countably generated modules of rank one over a complete discrete valuation ring.  

Let $H$ be a submodule of $G$.  We let $H^0$ denote $\{x \in G: ax \in H$ for some $a \in R\setminus \{0\}$\}.  Note that $H^0$ depends on $G$, but the choice of this module will always be clear from the context.  The importance of $H^0$ derives from the easily verified fact that the height of an element of $H$ is the same in $H^0$ as it is in $G$.  Note also that if $x_1, ..., x_n \in G$ and $a_1, ..., a_n, b_1, ... , b_n \in R\setminus \{0\}$ then $$\langle a_1x_1, ..., a_n x_n\rangle ^0 = \langle b_1x_1, ..., b_nx_n\rangle ^0,$$ a fact we will use repeatedly.  
	  
Suppose $G$ is a module over $R=\Z_p$ and $H$ is a submodule of $G$.  We say $H$ is a {\em nice} submodule of $G$ if for any ordinal $\alpha$, $$p^\alpha (G/H)= (p^\alpha G+H)/H,$$ or equivalently, every nonzero coset of $H$ has an element of maximal $p$-height \cite[p. 74]{R}. Note that $H$ is nice if and only if every nonzero coset of $H$ contains an element proper with respect to $H$.  

\subsection{Ulm's Theorem and Its Generalization in $L_{\infty\omega}$}

Let $G$ be any group and $p$ a prime. Then we let $G[p]$ denote $\{ x \in G: px = 0\}$, and write $p^\alpha G[p]$ for $(p^\alpha G)[p]$.
For each ordinal $\alpha$, we define the {\em Ulm invariant}, $$u_p (\alpha, G) = \dim  p^\alpha G [p] / p^{\alpha +1} G [p]$$ as a $\mathbb{Z}/(p)$-vector space. When $p$ is understood, we will write simply $u(\alpha, G)$.

\begin{theorem} \label{thm4} (Ulm) \cite{U}
If $G$ and $H$ are  countable reduced p-groups, then $G \cong H$ if and only if $u(\alpha , G) = u(\alpha , H)$ for all ordinals $\alpha$.
\end{theorem}

The proof is in Kaplansky \cite[Theorem 14]{Kap1} or Fuchs \cite[Theorem 77.3]{F}  and is based on a back-and-forth argument made possible by the countability hypothesis. This indicates that the Ulm invariants, or some modification of them, may be used to classify $p$-groups, not necessarily countable, up to partial isomorphism. This is in fact the case, as was proved by Barwise and Eklof. 
Define $$\hat{u}(\alpha, G)
= \min \{u(\alpha, G), \omega\}$$ 
for $G$ a group, $\alpha$ an ordinal.

\begin{theorem} \label{thm5} (Barwise and Eklof) \cite[Theorem 10]{BE}
For any two reduced p-groups $G$ and $H$, $ G \cong_p H$ if and only if $\hat{u}(\alpha, G)$ = $\hat{u}(\alpha, H)$ for every ordinal $\alpha$.
\end{theorem}

The proof of the ``if" direction is the same back-and-forth argument used in the proof of Ulm's Theorem. The ``only if" part assumes $G \equiv_\infty H$ and exhibits sentences of $L_{\infty \omega}$ that say ``$\hat{u}(\alpha, G) = n$" and ``$\hat{u}(\alpha, G) = \omega$".

If we add the invariant $$\hat{u}(\infty, G)
 = \min\{\dim p^\infty G [p], \omega\},$$ 
 this theorem classifies all torsion groups in $L_{\infty \omega}$. 

\subsection{Decomposition Bases and Warfield's Theorem}

In this paper we focus on the local case; the global case will be considered in \cite{JL5}. Every $p$-local group  may be considered as a module over  $\mathbb{Z}_p$.  More generally we may consider modules over a principal ideal domain.
	
If $G$ is a module over a principal ideal domain $R$ we say $X \subseteq G$ is a {\em decomposition set} if $X$ is an independent set of elements of infinite order and for all $x_1, \dots, x_n \in X$, $a_1, \dots ,a_n \in R$, and $p$ a prime in $R$, $$|a_1x_1+\cdots +a_nx_n|_p = \min_{1\le i \le n} \{|a_ix_i|_p\}.$$ We let $\langle X\rangle $ denote the submodule generated by $X$.  If $X$ is a decomposition set and $G/ \langle X\rangle $ is torsion, we say $X$ is a {\em decomposition basis} for $G$.

Here are some examples:
\begin{itemize}
\item[\rm{(i)}]The empty set forms a decomposition basis for any torsion module. 

\item[\rm{(ii)}]If $G$ is completely decomposable, say $G$ = $\bigoplus_{i \in I} G_i$, where each $G_i$ is of rank one, and if $ x_i \in G_i$ has infinite order, then $\{x_i\}_{i \in I}$ is a decomposition basis for $G$. 
\end{itemize}

Warfield \cite{W2} defined a class of $\mathbb{Z}_p$-modules consisting of all summands of simply presented modules.  (A $\Z_p$-module is called {\em simply presented} if it can be generated by a set of elements subject only to defining relations of the form $p^nx=0$ or $p^mx=y$ for generators $x$, $y$, $x\ne y$.)  These modules have come to be called {\em Warfield modules}. Warfield modules have decomposition bases \cite{HRW}.  Warfield classified these modules, extending previous work of Nunke \cite{N}, Hill \cite{H} and Baer \cite{Baer}.  Stanton \cite{S} extended these results to the global case with the use of additional invariants. 	

For $G$ a $\Z_p$-module with a decomposition basis $X$ and $e$ an equivalence class of Ulm sequences, we define  the Warfield invariant, $$w(e,G) =\text{ the cardinality of } \{x \in X: U(x) \in e\}.$$ Warfield proved that this is independent of the choice of $X$ and that these invariants, along with the Ulm invariants, serve to classify Warfield modules up to isomorphism.

\section{The Class and the Invariant}

\subsection{The Partial Decomposition Basis}
	
	We will define our class for modules over a principal ideal domain.  In order for our class to be closed under $L_{\infty \omega}$-equivalence, it must include more than just the modules with decomposition bases.  For example, $\prod_ \omega \mathbb{Z}  \cong _ p \bigoplus_\omega \mathbb{Z},$ although $\bigoplus_\omega \mathbb{Z}$ has a decomposition basis, and $\prod_\omega \mathbb{Z}$ does not \cite{Bar}.  The observation that a partial isomorphism involves finitely generated sets that may be extended motivates the following definition.  Let $G$ be a module over a principal ideal domain.  We say $\mathcal{C}$ is a {\em partial decomposition basis} for $G$ if
\begin{itemize}	
\item[\rm{(i)}]$\mathcal{C}$ is a nonempty collection of finite subsets of $G$,
	
\item[\rm{(ii)}]if $X \in \mathcal{C}$, then $X$ is a decomposition set, and
	
\item[\rm{(iii)}]if $X \in \mathcal{C}$ and $x\in G$, then there is a $Y \in \mathcal{C}$
			such that $X \subseteq Y$ and $x \in \langle Y\rangle ^0$.
\end{itemize}

It follows from this definition that the class of modules with partial decomposition bases is closed under $\equiv_\infty$. First we need a lemma.

\begin{lemma} \label{lemma1} Let $G$ and $H$ be modules over a principal ideal domain $R$ and $I$: $G \cong_pH$.  Then if $f \in I$, $f$ is height-preserving. 
\end{lemma}

\begin{proof}
Let $p$ be a prime in $R$. 
 We will prove by induction on $\alpha$ that if $x \in G$ and $|x|_p \ge \alpha$ then $|f(x)|_p \ge \alpha$ for all $f \in I$. Suppose $\alpha$ is a successor ordinal, say $\alpha = \beta + 1$, 
 $x \in G$ with $|x|_p \ge \alpha$ and $f \in I$.
 Then $x = py$ for some $y$ such that $|y|_p \ge \beta$.  
 Choose a $g \in I$ such that $f \subseteq g$ and $y \in \domain(g)$.  
 Then $f(x) = g(x) = pg(y)$ and $|g(y)|_p \ge \beta$ by the induction hypothesis, so $|f(x)|_p \ge \beta + 1 = \alpha$, as required. 
 Now suppose $\alpha$ is a limit ordinal.  Then for any $\beta < \alpha$, $x \in p^\beta G$, so by induction $|f(x)|_p \ge \beta$.  
 But since $\beta$ was arbitrary, we see that  $|f(x)|_p \ge \alpha$.  This completes the induction and proves $|f(x)|_p \ge |x|_p$. Applying the same argument to $f^{-1}$, we get $|f(x)|_p = |x|_p$. 
\end{proof}

\begin{theorem}\label{thm6} Let $G$ and $H$ be modules over a principal ideal domain. Suppose $G \cong _pH$ and $G$ has a partial decomposition basis.  Then so does $H$. Specifically, if $\C$ is a partial decomposition basis for $G$, and $I:G\cong_p H$, then
 $$\mathcal{C}' = \{f(X) : X \in \mathcal{C}, f \in I, X \subseteq domain(f)\}$$ 
 is a partial decomposition for $H$.

\end{theorem}
\begin{proof} 
Note that $\C'\ne \emptyset$ since $\C \ne \emptyset$ and $I \ne \emptyset$ and any $f \in I$ may be extended successively to each element of any $X \in \C$. Note also that $f(X)$ is an independent set of elements of infinite order since $f$ is an isomorphism. By the lemma, if $f\in I$, then $f$ is height-preserving, so $f(X)$ is also a decomposition set for any $X \in \C$. Now let $f(X) \in \mathcal{C}'$ and $x \in H$.  Choose $g \in I$ such that $f \subseteq g$ and $x \in \range(g)$, say $x=g(y)$.  Then since $X \in \mathcal{C}$, we may choose $Y \in \mathcal{C}$ such that $X \subseteq Y$ and $ y \in \langle Y\rangle ^0$.  Since $Y$ is finite, there is a $\tilde{g}\in I$ such that $g \subseteq\tilde{g}$ and $ Y \subseteq \range(\tilde{g})$.  Then $x=\tilde{g}(y) \in \langle \tilde{g}(Y)\rangle ^0$.  Finally $f(X) = \tilde{g}(X) \subseteq \tilde{g}(Y)$ and $\tilde{g}(Y) \in \mathcal{C}'$, so $\mathcal{C}'$ is a partial decomposition basis. 
\end{proof}

It is clear that if $G$ has a decomposition basis $X$, then the set of all finite subsets of $X$ forms a partial decomposition basis for $G$.  Thus we have defined a class closed under $ \equiv_\infty$ and containing any modules with a decomposition basis.  A module may have more than one partial decomposition basis, and the following theorem allows us to relate elements of two such bases.

\begin{theorem} \label{thm7} Let $G$ be a module over a principal ideal domain $R$ which has partial decomposition bases $\mathcal{C}$ and $\mathcal{C}'$. Let $X \in \mathcal{C}$ and $Y \in \mathcal{C}'$.  Then there are decomposition sets $X'$ and $Y'$ such that $X \subseteq X', Y \subseteq Y', X'$ and $Y'$ are unions of ascending chains of elements of $\mathcal{C}$ and $\mathcal{C}'$ respectively, and $\langle X'\rangle ^0 = \langle Y'\rangle ^0$. 
\end{theorem}

\begin{proof}
We will define, by induction on $i \in \omega$, sets $X_i \in \mathcal{C}$, $Y_i \in \mathcal{C}'$. Let $X_0 = X$ and $Y_0 = Y$. Suppose $X_i$ and $Y_i$ have been chosen. Choose $X_{i+1} \in \mathcal{C}$ such that $X_i \subseteq X_{i+1}$ and $Y_i \subseteq \langle X_{i+1}\rangle ^0$. Then choose similarly $Y_{i+1} \in \mathcal{C}'$ such that $Y_i \subseteq Y_{i+1}$ and $X_{i+1} \subseteq \langle Y_{i+1}\rangle ^0$. 

Let $X' = \bigcup_{i \in \omega}X_i$ and $Y' = \bigcup_{i \in \omega} Y_i$. We claim that
$\langle X'\rangle ^0 \subseteq \langle Y'\rangle ^0$. Let $z \in \langle X'\rangle ^0$.  
Then for some $n \in \mathbb{Z}$, 
$a, a_1,..., a_n \in R\setminus \{0\}$ and $x_1,..., x_n \in X'$, 
$az = a_1x_1+\cdots+a_nx_n$. 
 Choose $X_i \supseteq\{x_1,..., x_n\}$.  
 Then $ \{x_1,..., x_n\}  \subseteq  \langle Y_i\rangle ^0$, 
 so $\{bx_1,...,bx_n\} \subseteq \langle Y_i\rangle $ for some $b \in R\setminus\{0\}$.  
 But then $baz \in \langle Y_i\rangle  \subseteq \langle Y'\rangle $, 
so $z \in \langle Y'\rangle ^0$.  
 Similarly, $\langle Y'\rangle ^0 \subseteq \langle X'\rangle ^0$. 
\end{proof}

\subsection{The Invariant $\hat{w}(e, G)$}

We now introduce the invariant that will be used for the local classification theorem.  Thus we assume in the rest of this paper that all modules are modules over $\Z_p$ for some prime $p$ unless otherwise stated.

We define the analogue of Warfield's invariant for a module $G$ with a partial decomposition basis $\mathcal{C}$.  Let \begin{eqnarray*}\hat{w}_{\mathcal{C}}(e,G) = \text{ the maximum } n \text{ such that there is an } X \in \mathcal{C} \text{ and } x_1,..., x_n \in X \text{ such that } \\ U(x_i) \in e \text{ for }1\le i \le n, \text{ if such a maximum exists, and } \omega \text{  otherwise. }\end{eqnarray*}

\begin{theorem} \label{thm8} Let $G$ be a module with partial decomposition bases $\mathcal{C}$ and $\mathcal{C}'$.  Suppose for some equivalence class $e$ of Ulm sequences $\hat{w}_{\mathcal{C}}(e,G) \ge n$.  Then for any $Y \in \mathcal{C}'$ there is a $\tilde{Y} \in \mathcal{C}'$ such that $Y \subseteq \tilde{Y}$ and $\tilde{Y}$ contains at least $n$ elements $y$ such that $U(y) \in e$.
\end{theorem}
\begin{proof}  

Since $\hat{w}_\mathcal{C}(e,G) \ge n$, there is $X \in \mathcal{C}$ containing at least $n$ elements $x$ such that $U(x) \in e$.  Choose $X'$ and $Y'$ as in Theorem \ref{thm7}.  Consider the submodule $\langle X'\rangle ^0 = \langle Y'\rangle ^0$.  Since elements in $\langle X'\rangle ^0$ have the same height in $\langle X'\rangle ^0$ as they do in $G$, $X'$ is a decomposition set for $\langle X'\rangle ^0$, and consequently a decomposition basis.  Similarly, $Y'$ is a decomposition basis. Now $w(e,\langle X'\rangle ^0) \ge n$, since $X'$ contains at least $n$ elements $x$ such that $U(x) \in e$. But since as Warfield proved, $w(e,\langle X'\rangle ^0)$ 
 is independent of the choice of decomposition basis, and since $\langle X'\rangle^0 =\langle Y'\rangle^0$, $Y'$ also contains $n$ elements $y$ such that $U(y) \in e$.  Since $Y'$ is the union of an ascending chain of elements of $\mathcal{C}'$, we may choose an element of $\mathcal{C}'$ containing $Y$ and these $n$ elements as our required $\tilde{Y}$. 
\end{proof}

\begin{corollary} \label{cor1}
Let $G$ be a module with partial decomposition bases $\mathcal{C}$ and $\mathcal{C}'$.  Then for any equivalence class $e$ of Ulm sequences, $\hat{w}_\mathcal{C}(e,G) = \hat{w}_{\mathcal{C}'}(e,G)$.
\end{corollary}
\begin{proof}
Suppose $\hat{w}_\mathcal{C}(e,G) \ge n$.  Then by the theorem, with any $Y \in \mathcal{C}'$, $\hat{w}_{\mathcal{C}'}(e,G) \ge n$.  So if $\hat{w}_\mathcal{C}(e,G) = \omega$, $\hat{w}_{\mathcal{C}'}(e,G) = \omega,$ and if $\hat{w}_\mathcal{C}(e,G)$ is finite, $\hat{w}_{\mathcal{C}'}(e,G) \ge \hat{w}_\mathcal{C}(e,G)$. The equality follows by symmetry.
\end{proof}

The corollary shows that $\hat{w}_\mathcal{C}(e,G)$ is indeed an invariant, so we will drop the subscript $\mathcal{C}$ in what follows.  We may now restate Theorem \ref{thm8}. 

\begin{theorem} \label{thm8'}
Let $G$ be a module with partial decomposition basis $\mathcal{C}$.  Suppose for some equivalence class $e$ of Ulm sequences $\hat{w}(e,G) \ge n$.  Then for any $X \in \mathcal{C}$ there is a $\tilde{X} \in \mathcal{C}$ such that $X \subseteq \tilde{X}$ and $\tilde{X}$ contains at least $n$ elements $x$ such that $U(x) \in e$. 
\end{theorem}

We would like $\hat{w}(e,G)$ to be invariant under $\equiv _\infty$.  This is indeed the case, as we will show by an appropriate choice of partial decomposition basis. 

\begin{corollary}\label{cor2} Let $G$ and $H$ be modules such that $G$ has a partial decomposition basis $\mathcal{C}$ and $I: G \cong_p H$.  Then for any equivalence class $e$ of Ulm sequences, $\hat{w}(e,H)$ is defined and $\hat{w}(e,H) = \hat{w}(e,G)$.
\end{corollary}
\begin{proof}
Let $\C' = \{f(X): X \in \mathcal{C}, f \in I, X \subseteq \domain(f)\}$. By Theorem \ref{thm6}, $\C'$ is a partial decomposition basis for $H$, so $\hat{w}(e,H)$ is defined for all $e$.  Suppose for a given $e$, $\hat{w}(e,G) \ge n$.  Choose $X \in \mathcal{C}$ such that there are $x_1,..., x_n \in X$ with $U(x_i) \in e$ for $1\le i \le n$.  Choose $f \in I$ such that $X \subseteq \domain(f)$.   
Then $f(X) \in \mathcal{C}'$ and since $f$ is height preserving, $U(f(x_i)) \in e$ for $1 \le i \le n$. Thus $\hat{w}(e,H) \ge n$.  Conversely, if $\hat{w}(e,H) \ge n$ then there are $f(x_1),..., f(x_n) \in f(X)$  for some $f \in I$ and $X\in \mathcal{C}$ such that $U(f(x_i)) \in e$ for $1\le i \le n$.  Thus $U(x_i) \in e$ for $1 \le i \le n$ and $x_1,..., x_n \in X \in \mathcal{C}$, so $\hat{w}(e,G) \ge n$.
\end{proof}

If a module $G$ has a partial decomposition basis, then there exists a partial decomposition basis for $G$ which is closed under taking subsets and nonzero multiples:

\begin{theorem} \label{thm9} Suppose $G$ is a module with partial decomposition basis $\mathcal{C}$.  Then  $G$ also has a partial decomposition basis $\tilde{\C}$ such that for any $X \in \tilde{\C}$, $x_1,\dots,x_n \in X$ and $a_1, \dots, a_n  \in \Z \setminus \{0\}$, $\{a_1x_1,\dots,a_nx_n\} \in \tilde{\C}$. In particular, we may take
\begin{align*}
 \tilde{\mathcal{C}} &= \{\{a_1x_1,..., a_nx_n\}: \{x_1,..., x_n\} \subseteq X \text{ for some } X \in \mathcal{C}, a_1,\dots,a_n \in \Z \setminus \{0\}\}. 
\end{align*}
\end{theorem}

\begin{proof}
First, $\tilde{\mathcal{C}} \ne \emptyset$ since $\emptyset \in \tilde\C$. 
Clearly, if $X$ is a decomposition set, so is any nonzero multiple of any subset. 
If $x \in G$ and $\{a_1x_1,\dots,a_nx_n\} \in \tilde \C$, 
where $x_1,\dots,x_n \in X$ for some $X \in \C$,
choose $X' \supseteq X$ such that $x \in \langle X' \rangle^0$, 
and then if we let $Y= \{a_1x_1,\dots,a_nx_n\} \cup (X'\setminus \{x_1,\dots,x_n\}),$
$x \in \langle Y\rangle^0$ and $Y \in \tilde\C$.
Hence $\tilde{\C}$ is a partial decomposition basis. 
We claim that $\tilde \C$ satisfies the condition.
Let $\{a_1x_1,..., a_nx_n\} \in \tilde{\mathcal{C}}$, 
where $\{x_1,..., x_n\} \subseteq X$ for some $X \in \C$. 
Clearly, any subset of this set is also in $\tilde{\C}$, as is any set of multiples.  
\end{proof}

\section{The Local Classification Theorem}

\subsection{The Extension Theorem}

In order to prove the classification theorem, we will need to be able to extend certain functions for a back-and-forth argument.  One of the cases with which we are concerned is that of a function $f: S\rightarrow T$, where $S$ and $T$ are submodules of modules $G$ and $H$ respectively, and $x \in G$ has a multiple in $S$.  Our technique for doing this is based on Mackey's proof of Ulm's Theorem as given in Kaplansky \cite[Theorem 14]{Kap1}, which has been used in subsequent generalizations of Ulm's Theorem, including Warfield's work.  We will first need a lemma. 

\begin{lemma} \label{lemma2}
Let $S$ be a submodule of a module $G$ over a principal ideal domain and $\alpha$ any ordinal, $p$ any prime.  Let \begin{align*}S_\alpha &= S \cap p^\alpha G, \\S_\alpha ^* &= S_\alpha \cap \{x \in G : px \in p^{\alpha+2}G\}.\end{align*}  For $x \in S_\alpha ^*$, map $x$ onto $x-y+p^{\alpha+1}G[ p]$, where $y \in p^{\alpha+1}G$ and $px = py$.  This map is well-defined and has kernel $S_{\alpha+1}$.  The induced map $$U: S_\alpha^* / S_{\alpha +1} \rightarrow p^\alpha G[p]/p^{\alpha +1} G[p]$$ is an
isomorphism onto a submodule of $p^\alpha G[p]/p^{\alpha+1}G[p]$. The range of $U$ is not all of $p^\alpha G[p]/p^{\alpha+1}G[p]$ if and only of there exists in $p^\alpha G[p]$ an element of height $\alpha$ proper with respect to $S$. 
\end{lemma}

Kaplanky's proof of this lemma for $p$-groups \cite[Lemma 13]{Kap1} applies just as well to modules.   

We are now ready to prove the extension theorem.  The following proof is based on Lemma B of Barwise and Eklof \cite{BE}, a similar result for torsion groups, also based on Kaplansky's proof of Ulm's Theorem.

\begin{theorem}\label{thm10}
Let $G$ and $H$ be modules over a principal ideal domain $R$, $p$ a prime in $R$, and $\varphi$ an isomorphism between a finitely generated submodule $S$ of $G$ and a submodule $T$ of $H$ that preserves $p$-heights up to $\alpha$ for some ordinal $\alpha$.  Suppose that $\hat u_p(\sigma,G) \le \hat u_p(\sigma,H)$ for all $\sigma< \alpha$.  If $x \in G$ is proper with respect to $S$, $px \in S$, and $|x|_p + 1 < \alpha$, then for a suitable $y$ in $H$, $\varphi$ can be extended to an isomorphism $\varphi* : \langle S,x\rangle  \rightarrow \langle T,y\rangle $ that preserves $p$-heights up to $\alpha$.  
\end{theorem}

\begin{proof}Let $|x|_p = \sigma$.  Since any $x' \in x+S$ satisfies $px' \in S$, we may assume $x$ has been chosen such that $|px|_p > \sigma+1$ if this is possible for any proper element of the coset.  

\underline{Case 1:} $|px|_p = \sigma + 1 < \alpha$.  Then since $px \in S$ and $\varphi$ is height preserving up to $\alpha$ on $S$, $|\varphi (px)|_p = \sigma+1$ so we may choose $y \in H$ such that $\varphi (px) = py$ and $|y|_p = \sigma$. We claim $y \notin T$.

Suppose $y \in T$.  Then $y = \varphi (z)$ for some $z \in S$. $\varphi (pz) = py =\varphi (px)$, so $px = pz$ and $p(x-z) =0$. But $x-z \in x+S$ and so $|x-z|_p \le |x|_p = \sigma$, since $x$ is proper. Also $|x-z|_p \ge \min\{|x|_p,|z|_p\} = \sigma$, so $|x-z|_p = \sigma$. This says $x-z$ is proper with respect to $S$ and $|p(x-z)|_p = \infty > |px|_p$, contradicting the assumption that $x$ is chosen so that
 $|px|_p > \sigma +1$ if possible. 
 Thus $y \notin T$.

Now we claim $y$ is proper with respect to $T$. Suppose $|y+z|_p > \sigma$ for some $z \in T$. Then since $|y|_p = \sigma$, we must have $|z|_p = \sigma$. Also $|p(y+z)|_p > |y+z|_p \ge \sigma +1$, so since $\varphi$ preserves heights up to $\alpha > \sigma +1$, $|\varphi^{-1}(z)|_p = \sigma$ and $\sigma + 1 < |\varphi^{-1}(p(y+z))|_p = |p(x+\varphi^{-1}(z))|_p$. But $|x+\varphi^{-1}(z)|_p \le \sigma$ since $x$ is proper, and $$|x+\varphi^{-1}(z)|_p \ge \min\{|x|_p,|\varphi^{-1}(z)|_p\} = \sigma.$$ Thus $x+\varphi^{-1}(z)$ has height $\sigma$, is proper and its $p$-multiple has height $> \sigma+1$, again contradicting the choice of $x$. Thus $y$ is proper with respect to $T$.

\underline{Case 2:}
$|px|_p > \sigma+1$. Then $px=py'$ for some $y' \in p^{\sigma+1}G$. Then since $$|x|_p = \sigma < \sigma+1 \le |y'|_p,$$ $|x-y'|_p= \sigma$ and $x-y' \in G[p] $. Also $x-y'$ is proper with respect to $S$, since for any $z \in S$, $$|x-y'+z|_p = |x+z|_p \le |x|_p = \sigma.$$ Therefore, by Lemma \ref{lemma2} the range of $U: S_\sigma^*/S_{\sigma+1} \rightarrow p^\sigma G[p]/p^{\sigma +1}G[p]$ is not all of $p^\sigma G[p]/p^{\sigma +1}G[p]$. 
But $S$ is finitely generated so $S_\sigma^*/S_{\sigma+1}$ is also and so \begin{align*}\dim S_\sigma^*/S_{\sigma+1} &< \min\{\dim p^\sigma G[p]/p^{\sigma +1}G[p], \omega\} = \hat{u}_p(\sigma , G) \\&\le \hat{u}_p(\sigma , H) = \min\{\dim p^\sigma H[p]/p^{\sigma +1}H[p], \omega\}.\end{align*}

Since $\varphi$ is height-preserving up to $\alpha > \sigma +1$, $\varphi$ maps $S_\sigma$ onto $T_\sigma$ and $S_\sigma ^*$ onto $T_\sigma ^*$, hence $S_\sigma ^*/S_{\sigma+1}$ is mapped onto $T_\sigma ^*/T_{\sigma+1}$, giving them the same dimension. So the induced map $V: T_\sigma^*/T_{\sigma+1} \rightarrow p^\sigma H[p]/p^{\sigma +1}H[p]$ is not onto. Thus by Lemma \ref{lemma2} again, there is an element $w_1 \in H$ such that $pw_1 = 0$, $|w_1|_p = \sigma$, and $w_1$ is proper with respect to $T$.

Now choose $w_2 \in H$ such that $pw_2 = \varphi (px)$ and $|w_2|_p > \sigma$. Such a $w_2$ exists since $\varphi$ is height preserving up to $\alpha > \sigma +1$. Let $y = w_1 + w_2$. Then $py = pw_2 = \varphi(px)$ and $|y|_p = \min\{|w_1|_p, |w_2|_p\} = \sigma$. 
Also $y$ is proper with respect to $T$ since if $z \in T$, $|y+z|_p = |(w_1+z)+w_2|_p = |w_1 +z|_p \le \sigma$.

In either case we have produced a $y$ proper with respect to $T$, satisfying $|y|_p = \sigma$ and $py = \varphi(px)$. Define $$\varphi ^* : \langle S,x\rangle  \rightarrow \langle T,y\rangle  $$ by $x \mapsto y$ and $\varphi ^*\restriction S = \varphi$. $\varphi ^*$ is clearly a well-defined isomorphism. Using the fact that $R$ is a B\' ezout domain, it may be verified that every element $z$ of $\langle S,x\rangle $ is in $S$ or for some $r$ not divisible by $p$, $rz \in x+S$. We already know that $\varphi ^*$ preserves heights of elements of $S$ up to $\a$, so we need only consider $z \in \langle S,x\rangle $ such that $rz = s+x$ for some $r \in R\setminus\{0\}$, $s \in S$, where $p$ does not divide $r$. Then $$|z|_p = |rz|_p = |s+x|_p =\min\{|s|_p,|x|_p\}$$ since $x$ is proper, and similarly, $$|\varphi ^*(z)|_p = |r\varphi ^* (z)|_p = |\varphi ^*(s+x)|_p = |\varphi (s) +y|_p =\min\{|\varphi(s)|_p, |y|_p\} =\min\{|s|_p, |x|_p\}.$$
\end{proof}

\subsection{The Classification Theorem}

Originally we followed  Warfield (see the unpublished version of \cite{W2}) in our approach. In his classification theorem for countably generated modules, he first proves the case for modules over a complete discrete valuation ring and then extends it to a general module using a mapping. The following lemma is based on Hunter and Richman \cite{HR}, and gives nice submodules without any assumption of completeness. This allows us to simplify and shorten the proof of the classification theorem considerably. In particular, in what follows we may assume that all modules are modules over $\Z_p$ for some fixed prime $p$.

\begin{lemma}\label{lemma4_7}\cite{JL2}
Let $G$ be a module with partial decomposition basis $\mathcal{C}$ and $Y \subseteq X \in \mathcal{C}$. Then $\langle Y\rangle$ is nice in $G$.
\end{lemma}

\begin{proof}
Let $x\in G$. Then there is an $X'\in\C$ such that $X\subseteq X'$ and $x\in\langle X'\rangle^0$. Clearly $X'$ is
a decomposition basis for the module $\langle X'\rangle^0$. Since $Y$ is a finite subset of $X'$, $\langle Y\rangle$ is nice in $\langle X'\rangle^0$ by \cite[Lemma 8.1]{HR}, hence the coset $x+\langle Y\rangle$ has an element $y$ of
maximal height in $\langle X'\rangle^0$. The height of $y$ is the same in $\langle X'\rangle^0$ as it is in $G$, so the proof is complete.
\end{proof}

This lemma, along with the following lemma of Warfield, will make all of the submodules of interest nice. 

\begin{lemma} \label{lemma3}\cite{W1}
If $S$ is a nice submodule of a module $G$ and $K$ is a submodule containing $S$ such that $K/S$ is finitely generated and torsion, then $K$ is also a nice submodule.
\end{lemma}

We will need the following lemma, which is an analogue in $L_{\infty \omega}$ of a lemma Warfield proved for countably generated modules \cite[Corollary 6.6]{W2}.

\begin{lemma}\label{lemma4}
Let $G$ be a reduced module of finite rank with a decomposition basis $\{x_1,\dots, x_n\}$. Then there is a reduced torsion module $T$ and submodules $G_i$, for $1 \le i \le n$, of $G$ such that $$G \oplus T \cong_p G_1 \oplus \cdots \oplus G_n$$ and $x_i \in G_i$.
In fact, the set $I$ of all height-preserving isomorphisms between finitely generated submodules of $G\oplus T$ and $G_1\oplus...\oplus G_n$ extending the canonical map $\langle x_1,...,x_n\rangle \to \langle x_1,...,x_n\rangle$ satisfies the back-and-forth property.
\end{lemma}
\begin{proof}
For each $i$, let $$G_i = \langle x_i\rangle ^0 \subseteq G.$$ Let $T$ be the direct sum of $n-1$ copies of $tG$, the torsion submodule of $G$. Then $G \oplus T$ and $G_1 \oplus \cdots \oplus G_n$ have the same Ulm invariants since the torsion submodule of each is isomorphic to $tG^n$. Let $S$ and $L$ be the submodules of $G\oplus T$ and $G_1 \oplus \cdots \oplus G_n$ respectively generated by the $x_i$, $1 \le i \le n$. Let $\varphi : S \rightarrow L$ be the isomorphism such that $$\varphi (x_i) = x_i \text{ for } 1 \le i \le n.$$ Then, since the height of multiples of $x_i$, for $1 \le i \le n$, is the same in $G$, $G \oplus T$ or $G_1 \oplus \cdots \oplus G_n$, and since $\{x_1, \dots , x_n \}$ is a decomposition basis for $G \oplus T$ or $G_1 \oplus \cdots \oplus G_n$, $\varphi$ is height preserving.

Let  $I$ be the set of all $f: A \rightarrow B$ such that 
\begin{itemize} 
\item[\rm{(i)}] $A \text { and } B \text{ are finitely generated submodules of } G \oplus T 
\text{ and } G_1 \oplus \cdots \oplus G_n \text{ respectively, }$
\item[\rm{(ii)}] $S \subseteq A , L \subseteq B$, 
\item[\rm{(iii)}] $f$ is a height-preserving isomorphism and 
\item[\rm{(iv)}] $\varphi \subseteq f$.
\end{itemize} 
$I$ is not empty since it contains the function $\varphi$. 
We will prove that $I$ has the back-and-forth property. 
Let $f \in I$, $z+w \in G \oplus T$, where $z \in G$ and $w \in T$. 
Note that $G\oplus T$ is torsion over $S$, so $p^m(z+w) \in S \subseteq A$ for some $m$. 
We claim by induction on $k$, $0 \le k \le m$, that 
there is an $f_k:A_k \rightarrow B_k$ in $I$ such that 
$f \subseteq f_k$ and $p^{m-k}(z+w) \in A_k$. 
Define $f_0=f$. 
Suppose $f_k$ has been defined for $k<m$. 
By Lemma \ref{lemma4_7}, $S$ is nice in $G \oplus T$, and since $A_k$ is torsion over $S$, 
$A_k$ is nice by Lemma \ref{lemma3}. 
Choose $x$ proper in $p^{m-k-1}(z+w)+A_k$. 
Then $px \in A_k$. 
If $x \in S$, let $f_{k+1}=f_k$. 
Suppose $x \notin S$, in particular $x \ne 0$. 
Then $|x| + 1 < \length(G) + 1$. 
By Theorem \ref{thm10}, there is an isomorphism 
$f_{k+1}: \langle A_k,x\rangle \rightarrow \langle B_k,y\rangle$ 
that preserves heights up to $\length(G)+1 > \length(G) = \length(G\oplus T)$. 
Then $f_{k+1} \in I$ and $p^{m-k-1}(z+w) \in x+A_k \subseteq A_{k+1}=\langle A_k,x\rangle$. 
This completes the induction. 
Let $g=f_m$. 
Then $z+w \in \domain(g), f \subseteq g$ and $g\in I$, as desired.

Similarly, if $z_1 + \dots +z_n \in G_1 \oplus \cdots \oplus G_n$, we may choose an $m \ge 0$ such that $p^mz_i \in \langle x_i\rangle $ for $1 \le i \le n$. Then $p^m(z_1 + \dots + z_n) \in \langle x_1, \dots , x_n\rangle $, so we may extend $f^{-1}$ to $\langle B,z_1 + \dots + z_n\rangle $. Thus $I$ has the back-and-forth property and so is a partial isomorphism.

\end{proof}

\begin{lemma} \label{lemma5} \cite{JL2}
Let $G$ be a module, $X$ a decomposition basis for $G$ and $S$ a finitely generated submodule of $G$ such that $S \cap \langle X\rangle  = \langle S \cap X\rangle $. Then for every $y \in X$ such that $y \notin S$, there is an $n \in \omega$ such that for all $r \in \Z_p$ and $s \in S$, $$|rp^ny+s| = \min\{|rp^ny|, |s|\}.$$
\end{lemma}
\begin{proof}
A detailed proof, which can be found in \cite[ Lemma 4.4]{JL2} will be outlined here for the sake of completeness.
The conclusion is immediate if $U(y) \sim (\infty,\infty,\dots)$, so we may assume $U(y) \not\sim (\infty,\infty,\dots)$.
Let $S \cap X = \{x_1, \dots , x_m\}$ where for some $k \le m$, $U(x_i) \sim (\infty, \infty, \ldots)$ if and only if $i > k$. Let $H = (S \oplus \langle y\rangle)^0$ with divisible part $D$. 
By \cite[Theorem 6]{Kap2}, we may write $H=R\oplus D$ where $R$ is a reduced module with decomposition basis $\{ y, x_1, \ldots, x_k\}$. By Lemma \ref{lemma4} there is an
$I: R\oplus T\cong_p H_0\oplus\ldots\oplus H_k$
where $T$ is torsion, $y\in H_0$ and $x_i\in H_i$ for $i=1,\ldots, k$. For a map $f:A\to B$ in $I$ define
$f':  A\oplus D  \to B\oplus D$ by $(a,x) \mapsto  (f(a),x)$.

Since $I'=\{ f' : f\in I\}$ has the back-and-forth property, the finitely generated $S\oplus\langle y\rangle$ is contained in the domain of some $f'\in I'$. Then $\pi_{H_0} f'(S)$ is finitely generated. It is also torsion, since any element of $S$ has a multiple in $\langle S\cap X\rangle$. Hence its elements
take on  only a finite number of heights. Since the elements
$p^nf'(y)$ ($n\in\omega$) have infinitely many heights, there is an $n\in\omega$ such that $|rp^n y|\neq |s|$ for  all $0\neq r\in \Z_p$ and $s \in \pi_{H_0} f'(S)$.

Now let $s \in S$, $r \in \Z_p$ and write $f'(s)$ as an element of $H_0\oplus\dots \oplus H_k\oplus D$. Then we may show that  
$|f'(rp^ny + s)| = \min\{|f'(rp^ny)|, |f'(s)|\}$
and the conclusion follows since $f'$ is height-preserving.
\end{proof}

Now we are ready for the classification theorem.

\begin{theorem} \label{thm11} 
Let $G$ and $H$ be modules with partial decomposition bases. Then $G \cong_p H$ if and only if for every $\alpha$, an ordinal or $\infty$, and equivalence class $e$ of Ulm sequences $\hat{u} (\alpha, G) = \hat{u}(\alpha, H)$ and $\hat{w}(e,G) = \hat{w}(e,H)$. In that case, if $\mathcal{C}$ and $\C'$ are partial decomposition bases of $G$ and $H$ respectively as in Theorem \ref{thm9} and if $X_0 \in \mathcal{C}$, $Y_0 \in \mathcal{C}'$ and $f_0 : X_0 \rightarrow Y_0$ is a bijection such that $U(x) = U(f_0 (x))$ for all $x \in X_0$, then $I:G\cong_p H$ may be chosen to be the set of all maps $f:S\to T$ for which there exist $X\in\C$ and $Y\in\C'$ satisfying the following properties:
\begin{itemize}
    	\item [\rm{(i)}] $S$ and $T$ are finitely generated submodules of $G$ and $H$ respectively;
	\item [\rm{(ii)}] $f$ is a height-preserving isomorphism;
	\item [\rm{(iii)}] $X \subseteq S \subseteq \langle X\rangle ^0$ and $Y \subseteq T \subseteq \langle Y \rangle^0$;
	\item [\rm{(iv)}] $f(X) =Y$;
	\item [\rm{(v)}] $X_0 \subseteq X$ and $f \restriction X_0 = f_0$.
\end{itemize}
\end{theorem}
\begin{proof}
First, suppose $G \cong_p H$. Then their $\hat w$ invariants agree by Corollary \ref{cor2} and their $\u$ invariants agree by \cite[Lemma 2.2]{BE}.

Now suppose $G$ and $H$ have the same $\hat u$ and $\hat w$ invariants. Let $\C$ and $\C'$ be the respective partial decomposition bases which are closed under taking subsets and multiples as in Theorem \ref{thm9}. Notice that a map $f_0:X_0\to Y_0$ as stated in the theorem always exists; take for instance $f=\emptyset$ with $X_0=Y_0=\emptyset$. Note that if $f_0=\emptyset$, condition (v) is clearly vacuous. 

Let $I$ be the set of $f: S \rightarrow T$ such that there exists $X \in \mathcal{C}$ and $Y \in \C'$ such that $f, S, T, X$ and $Y$ satisfy (i) through (v). The set $I$ is not empty since it contains the extension $\tilde{f}_0:\langle X_0\rangle\to\langle Y_0\rangle$ of $f_0$. We will prove that $I$ has the back-and-forth property. Let $f: S \rightarrow T$ be an element of $I$ with associated $X$ and $Y$ and suppose $a \in G \setminus S$. Choose $X' \in \mathcal{C}$ such that $X \subseteq X'$ and $a \in \langle X'\rangle ^0$. 

First we claim $f$ may be extended to some nonzero multiple of any $x \in X'$, $x \notin S$. By taking subsets we get $X'' = X \cup \{x\} \in \C$. 
Let $e=[U(x)]$, the equivalence class of $U(x)$. Since $f$ is a height-preserving isomorphism, $\w(e,\langle X\rangle^0) = \w(e,\langle Y\rangle^0)$. 
Since $$\w(e,H) =\w(e,G) \ge \w(e,\langle X \cup \{x\}\rangle^0) > \w(e,\langle X \rangle^0),$$by Theorem \ref{thm8'} there is a $Y' \in \C'$ with $y \in Y' \setminus Y$ and $U(y)\sim U(x)$. Let $Y''=Y\cup\{y\} \in \C'$.
Apply Lemma \ref{lemma5} to the module $\langle S,x\rangle ^0$, the submodule $S$ and the decomposition basis $X''$. This is possible since $S\cap \langle X''\rangle  = \langle X\rangle  = \langle S \cap X''\rangle$ and $x \notin S$. Apply it similarly to $\langle T,y\rangle^0$, $T$ and $Y''$. 
Thus there is an $n$ such that for all $s \in S$, $t \in T$ and $r \in \Z_p$, \begin{align*}|rp^nx + s| &= \min\{|rp^nx|, |s|\} \text{ and} \\|rp^ny+t| &=\min\{|rp^ny|,|t|\}\end{align*} 

Let $x' = p^nx$ and $y'=p^ny$.
Since $U(x')\sim U(y')$, there are $\tilde x$ and $\tilde y$, nonzero multiples respectively of $x'$ and $y'$, such that $U(\tilde x) = U(\tilde y)$.
Define $g: \langle S,\tilde x\rangle  \to \langle T,\tilde y\rangle$ extending $f$ by $g(\tilde x) = \tilde y$. 
It is easily verified that $g$ satisfies all conditions (i) - (v) on $I$ with $X \cup \{\tilde x\}$ and $Y \cup \{\tilde y\}$. 

This proves that we may extend $f$ successively to some nonzero multiple of each element of $X'$, specifically that there is a $g: S^\ast \rightarrow T^\ast$ with related $X^\ast$ and $Y^\ast$ satisfying (i) - (v) such that $a \in \langle X'\rangle^0 = \langle X^\ast \rangle^0$. Since $X^\ast \subseteq S^\ast$, for some $m$, $p^ma \in S^\ast$. It suffices to prove that we may extend $g$ to $a$ in the case $pa \in S^\ast$ and $a \not\in S^\ast$.

So suppose $pa\in S^\ast$. By Lemma \ref{lemma4_7} $\langle X^\ast\rangle$ is nice since $X^\ast \in \C$. Since $S^\ast \subseteq \langle X^\ast \rangle ^0$, $S^\ast/\langle X^\ast\rangle$ is torsion. $S^\ast/\langle X^\ast\rangle$ is finitely generated since $S^\ast$ is. It then follows that $S^\ast$ is nice by Lemma \ref{lemma3}, so we may assume $a$ is proper with respect to $S^\ast$. If $|a| \ne \infty$,  by Theorem \ref{thm10}, we may extend $g$ to a height-preserving $g^*: \langle S^\ast,a\rangle  \rightarrow \langle T^\ast,b\rangle $ for some $b \in H$. Now suppose $|a|=\infty$. Then $|g(pa)|=\infty$ and so $g(pa)=pb'$ for some $b' \in H$ with $|b'|=\infty$. We claim that there is a $b \in H\setminus T^\ast$ such that $pa=pb$ and $|b|=\infty$. If $b' \not\in T^\ast$, we are done, so suppose $b' \in T^\ast$, say $b'=g(z)$ for some $z \in S^\ast$. Then $g(pa-pz)=0$, so $p(a-z)=0$ and hence $a-z \in p^\infty G[p]$. Thus $\dim(p^\infty H \cap T^\ast)[p]=\dim(p^\infty G\cap S^\ast)[p] < \dim(p^\infty G\cap \langle S^\ast,a\rangle)[p] \leq \u(\infty, G)=\u(\infty, H)$. It follows that there is a $\tilde b \in (p^\infty H)[p] \setminus(p^\infty H\cap T^\ast)[p]$. Then $\tilde{b}\not\in T$, $|\tilde b|=\infty$ and $p\tilde b = 0$. Let $b = b'+\tilde b$. Then $pb=g(pa)$, $b \not\in T^\ast$ and $|b|=\infty$. 

In either case, we may extend $g$ to a height-preserving isomorphism $$g^*: \langle S^\ast,a\rangle  \rightarrow \langle T^\ast,b\rangle $$ for some $b \in H$. Then $g^*$ satisfies conditions (i) and (ii) immediately, and if we leave $X^\ast$ and $Y^\ast$ unchanged, (iv) and (v) are also satisfied. Since $a \in \langle S^\ast\rangle ^0 \subseteq \langle X^\ast\rangle ^0$, $\langle S^\ast,a\rangle  \subseteq \langle X^\ast\rangle ^0$, and similarly $\langle T^\ast,b \rangle \subseteq \langle Y^\ast\rangle$, so (iii) is satisfied.

This proves that for any $a \in G$, any $f \in I$ can be extended to a $g^* \in I$ such that $a \in \domain(g^*)$. The extension to elements of $H$ follows by symmetry.
Consequently, the set $I$ has the back-and-forth property and we conclude that $G\cong_p H$.

\end{proof}

\bibliographystyle{amsplain}

\end{document}